\documentclass[11pt]{article}
\usepackage{amsfonts, amssymb, amsmath, amsthm, xcolor, graphicx}

\begin{document}

\theoremstyle{plain}
\newtheorem{thm}{Theorem}[section]
\newtheorem{cor}[thm]{Corollary}
\newtheorem{con}[thm]{Conjecture}
\newtheorem{cla}[thm]{Claim}
\newtheorem{lm}[thm]{Lemma}
\newtheorem{prop}[thm]{Proposition}
\newtheorem{example}[thm]{Example}

\theoremstyle{definition}
\newtheorem{dfn}[thm]{Definition}
\newtheorem{alg}[thm]{Algorithm}
\newtheorem{prob}[thm]{Problem}
\newtheorem{rem}[thm]{Remark}

\renewcommand{\baselinestretch}{1.1}

\title{\bf An Analogue of the Hilton-Milner Theorem  for weak compositions}
\author{
Cheng Yeaw Ku
\thanks{ Department of Mathematics, National University of
Singapore, Singapore 117543. E-mail: matkcy@nus.edu.sg} \and Kok
Bin Wong \thanks{
Institute of Mathematical Sciences, University of Malaya, 50603
Kuala Lumpur, Malaysia. E-mail:
kbwong@um.edu.my.} } \maketitle

\begin{abstract}\noindent
Let $\mathbb N_0$ be the set of non-negative integers, and let $P(n,l)$ denote the set of all weak compositions of $n$ with $l$ parts, i.e., $P(n,l)=\{ (x_1,x_2,\dots, x_l)\in\mathbb N_0^l\ :\ x_1+x_2+\cdots+x_l=n\}$. For any element $\mathbf u=(u_1,u_2,\dots, u_l)\in P(n,l)$, denote its $i$th-coordinate by $\mathbf u(i)$, i.e., $\mathbf u(i)=u_i$.  A family $\mathcal A\subseteq P(n,l)$ is said to be $t$-intersecting if $\vert \{ i \ :\ \mathbf u(i)=\mathbf v(i)\} \vert\geq t$ for all $\mathbf u,\mathbf v\in \mathcal A$. A family $\mathcal A\subseteq P(n,l)$ is said to be trivially $t$-intersecting  if there is a  $t$-set $T$ of $\{1,2,\dots,l\}$ and elements $y_s\in \mathbb N_0$ ($s\in T$) such that $\mathcal{A}= \{\mathbf u\in P(n,l)\ :\ \mathbf u(j)=y_j\ \textnormal{for all}\ j\in T\}$.  We prove that given any positive integers $l,t$ with $l\geq 2t+3$, there exists a constant $n_0(l,t)$ depending only on $l$ and $t$, such that for all $n\geq n_0(l,t)$, if $\mathcal{A} \subseteq P(n,l)$ is non-trivially $t$-intersecting then
\begin{equation}
\vert \mathcal{A} \vert\leq {n+l-t-1 \choose l-t-1}-{n-1 \choose l-t-1}+t.\notag
\end{equation}
Moreover, equality holds if and only if there is a $t$-set $T$ of $\{1,2,\dots,l\}$ such that
\begin{equation}
\mathcal A=\bigcup_{s\in \{1,2,\dots, l\}\setminus T} \mathcal A_s\cup \left\{ \mathbf  q_i\ :\ i\in T \right\},\notag
\end{equation}
where
\begin{align}
\mathcal{A}_s & =\{\mathbf u\in P(n,l)\ :\ \mathbf u(j)=0\ \textnormal{for all}\ j\in T\ \textnormal{and}\ \mathbf u(s)=0\}\notag
\end{align}
and $\mathbf q_i\in P(n,l)$ with $\mathbf q_i(j)=0$ for all $j\in \{1,2,\dots, l\}\setminus \{i\}$  and $\mathbf q_i(i)=n$.
\end{abstract}

\bigskip\noindent
{\sc keywords:}  cross-intersecting family, Hilton-Milner, Erd{\H
o}s-Ko-Rado, weak compositions

\section{Introduction}

Let $[n]=\{1, \ldots, n\}$, and let ${[n] \choose k}$ denote the
family of all $k$-subsets of $[n]$. A family $\mathcal{A}$ of subsets of $[n]$ is $t$-{\em intersecting} if $|A \cap B| \ge t$ for all $A, B \in \mathcal{A}$. One of the most beautiful results in extremal combinatorics is
the Erd{\H o}s-Ko-Rado theorem.

\begin{thm}[Erd{\H o}s, Ko, and Rado \cite{EKR}, Frankl \cite{Frankl}, Wilson \cite{Wilson}]\label{EKR} Suppose $\mathcal{A} \subseteq {[n] \choose k}$ is $t$-intersecting and $n>2k-t$. Then for $n\geq (k-t+1)(t+1)$, we have
\begin{equation}
\vert \mathcal{A} \vert\leq {n-t \choose k-t}.\notag
\end{equation}
Moreover, if $n>(k-t+1)(t+1)$ then equality holds if and only if $\mathcal{A}=\{A\in {[n] \choose k}\ :\ T\subseteq A\}$ for some $t$-set $T$.
\end{thm}

In the celebrated paper \cite{AK}, Ahlswede and Khachatrian extended the Erd{\H o}s-Ko-Rado theorem by determining the structure of all  $t$-intersecting set systems of maximum size for all possible $n$ (see also \cite{Bey, FT, Kee, Ku_Wong3, MT, Pyber, Toku, WZ0} for some related results). There have been many recent results showing that a version of the Erd{\H o}s-Ko-Rado theorem holds for combinatorial objects other than set systems. For example, an analogue of the Erd{\H o}s-Ko-Rado theorem for the Hamming scheme is proved in \cite{Moon}. A complete solution for the $t$-intersection problem in the Hamming space is given in \cite{AK2}. Intersecting families of permutations were initiated by Deza and Frankl in \cite{DF}. Some recent work done on this problem and its variants can be found in \cite{B2, BH,CK, E, EFP, GM, KL, KW, LM, LW, WZ}. The investigation of the Erd{\H os}-Ko-Rado  property for graphs started in \cite{HT}, and gave rise to \cite{B, B3, HS, HST, HK, W}. The Erd{\H o}s-Ko-Rado type results also appear in vector spaces \cite{CP, FW},  set partitions \cite{KR, Ku_Wong, Ku_Wong2} and weak compositions \cite{Ku_Wong4, Ku_Wong5}.

For a family $\mathcal{A}$ of $k$-subsets, $\mathcal{A}$ is said
to be {\em trivially} $t$-intersecting if there exists a $t$-set $T=\{x_1,\dots, x_t\}$ such that all members of $\mathcal{A}$ contains $T$. The
Erd{\H o}s-Ko-Rado theorem implies that a $t$-intersecting family
of maximum size must be trivially $t$-intersecting when $n$ is sufficiently large in terms of $k$ and $t$.

Hilton and Milner \cite{HM} proved a strengthening of the Erd{\H
o}s-Ko-Rado theorem for $t=1$ by determining the maximum size of a
non-trivial 1-intersecting family. A short and elegant proof was
later given by Frankl and F{\" u}redi \cite{FF} using the
shifting technique.

\begin{thm}[Hilton-Milner]\label{Hilton-Milner}
Let $\mathcal{A} \subseteq {[n] \choose k}$ be a non-trivial
1-intersecting family with $k \ge 4$ and $n>2k$. Then
\[ |\mathcal{A}| \le {n-1 \choose k-1} - {n-k-1 \choose k-1}+1.
\]
Equality holds if and only if
\[ \mathcal{A} = \left\{X \in {[n] \choose k}: x \in X, X \cap Y \not
= \emptyset\right\} \cup \{Y\}\]
for some $k$-subset $Y \in {[n] \choose k}$ and $x \in X
\setminus Y$.
\end{thm}

In this paper, we prove an analogue of Theorem \ref{Hilton-Milner}  for weak compositions with fixed number of parts. Let $\mathbb N_0$ be the set of non-negative integers, and let $P(n,l)$ denote the set of all weak compositions of $n$ with $l$ parts, i.e., $P(n,l)=\{ (x_1,x_2,\dots, x_l)\in\mathbb N_0^l\ :\ x_1+x_2+\cdots+x_l=n\}$. For any element $\mathbf u=(u_1,u_2,\dots, u_l)\in P(n,l)$, denote its $i$th-coordinate by $\mathbf u(i)$, i.e., $\mathbf u(i)=u_i$.  A family $\mathcal A\subseteq P(n,l)$ is said to be $t$-\emph{intersecting} if $\vert \{ i \ :\ \mathbf u(i)=\mathbf v(i)\} \vert\geq t$ for all $\mathbf u,\mathbf v\in \mathcal A$.

A family $\mathcal A\subseteq P(n,l)$ is said to be \emph{trivially} $t$-intersecting  if there is a  $t$-set $T$ of $\{1,2,\dots,l\}$ and elements $y_s\in \mathbb N_0$ ($s\in T$) such that $\mathcal{A}= \{\mathbf u\in P(n,l)\ :\ \mathbf u(j)=y_j\ \textnormal{for all}\ j\in T\}$. If $y_j=0$ for all $j\in T$, then $\mathcal A$ is said to be \emph{strong-trivially} $t$-intersecting.
It has been shown that a $t$-intersecting family
of maximum size must be strong-trivially $t$-intersecting when $n$ is sufficiently large in terms of $l$ and $t$ \cite[Theorem 1.2]{Ku_Wong4}.
Our main result is the following.

\begin{thm}\label{thm_main} Given any positive integers $l,t$ with $l\geq 2t+3$, there exists a constant $n_0(l,t)$ depending only on $l$ and $t$, such that for all $n\geq n_0(l,t)$, if $\mathcal{A} \subseteq P(n,l)$ is non-trivially $t$-intersecting then
\begin{equation}
\vert \mathcal{A} \vert\leq {n+l-t-1 \choose l-t-1}-{n-1 \choose l-t-1}+t.\notag
\end{equation}
Moreover, equality holds if and only if there is a $t$-set $T$ of $[l]$ such that
\begin{equation}
\mathcal A=\bigcup_{s\in [l]\setminus T} \mathcal A_s\cup \left\{ \mathbf  q_i\ :\ i\in T \right\},\notag
\end{equation}
where
\begin{align}
\mathcal{A}_s & =\{\mathbf u\in P(n,l)\ :\ \mathbf u(j)=0\ \textnormal{for all}\ j\in T\ \textnormal{and}\ \mathbf u(s)=0\}\notag
\end{align}
and $\mathbf q_i\in P(n,l)$ with $\mathbf q_i(j)=0$ for all $j\in [l]\setminus \{i\}$  and $\mathbf q_i(i)=n$.
\end{thm}

Note that
\begin{equation}
\bigcup_{s\in [l]\setminus T} \mathcal A_s=P(n,l)\setminus\{\mathbf u\in P(n,l)\ :\ \mathbf u(j)=0\ \textnormal{for all}\ j\in T\ \textnormal{and}\ \mathbf u(s)\neq 0 \ \textnormal{for all}\ s\in [l]\setminus T\}.\notag
\end{equation}
Therefore
\begin{equation}
\left\vert\bigcup_{s\in [l]\setminus T} \mathcal A_s\right\vert={n+l-t-1 \choose l-t-1}-{n-1 \choose l-t-1}.\notag
\end{equation}

\section{Certain intersecting conditions}

\begin{lm}\label{lm_pre_intersect3} Let $y_1, y_2,\dots, y_r\in \mathbb N_0$ and $m$ be a positive integer. Let $r\geq 3$ and
\begin{align}
\mathcal C & =\{ \mathbf u\in P(m,r)\ :\ \mathbf u(i)= y_i\ \textnormal{for some $i\in [r-2]$ or}\ \mathbf u(r-1)=0\ \textnormal{or}\ \mathbf u(r)=0\},\notag\\
\mathcal D & =\{ \mathbf u\in P(m,r)\ :\ \mathbf u(i)=y_i\ \textnormal{for some $i\in [r]$}\}.\notag
\end{align}
Then $\vert \mathcal D\vert\leq \vert \mathcal C\vert$.
\end{lm}

\begin{proof} Let
\begin{align}
\mathcal E_0 & =\{ \mathbf u\in P(m,r)\ :\ \mathbf u(i)=y_i\ \textnormal{for some $i\in [r-2]$}\}\notag\\
\mathcal C_1 & =\{ \mathbf u\in P(m,r)\ :\ \mathbf u(i)\neq y_i\ \textnormal{for all $i\in [r-2]$ and ($\mathbf u(r-1)=0$ or $\mathbf u(r)=0$)}\}\notag\\
\mathcal D_1 & =\{ \mathbf u\in P(m,r)\ :\ \mathbf u(i)\neq y_i\ \textnormal{for all $i\in [r-2]$ and ($\mathbf u(r-1)=y_{r-1}$ or $\mathbf u(r)=y_r$)}\}.\notag
\end{align}
Then $\mathcal C=\mathcal E_0 \cup \mathcal C_1$ and $\mathcal D=\mathcal E_0 \cup \mathcal D_1$. Note that $\mathcal E_0\cap \mathcal C_1=\varnothing=\mathcal E_0\cap \mathcal D_1$.  So, it is sufficient to show that $\vert \mathcal D_1\vert \leq \vert \mathcal C_1\vert$. If $y_{r-1}=0=y_r$, then $\mathcal D_1= \mathcal C_1$. Suppose $y_{r-1}+y_r>0$.

 Let $S_i=[m]\setminus \{y_i\}$ and $S=S_1\times S_2\times \cdots S_{r-2}$. For each $(d_1,\dots,d_{r-2})\in  S$, let
\begin{align}
\mathcal C_1(d_1,\dots,d_{r-2}) & =\{ \mathbf u\in P(m,r)\ :\ \mathbf u(i)=d_i\ \textnormal{for all $i\in [k-1]$ and ($\mathbf u(r-1)=0$ or $\mathbf u(r)=0$)}\}\notag\\
\mathcal D_1(d_1,\dots,d_{r-2}) & =\{ \mathbf u\in P(m,r)\ :\ \mathbf u(i)=d_i\ \textnormal{for all $i\in [k-1]$ and ($\mathbf u(r-1)=y_{r-1}$ or $\mathbf u(r)=y_r$)}\}.\notag
\end{align}
Note that
\begin{align}
\mathcal C_1 & =\bigcup_{(d_1,\dots,d_{r-2})\in  S}\mathcal C_1(d_1,\dots,d_{r-2})\notag\\
\mathcal D_1 & =\bigcup_{(d_1,\dots,d_{r-2})\in  S}\mathcal D_1(d_1,\dots,d_{r-2}),\notag
\end{align}
Furthermore, $\mathcal C_1(d_1,\dots,d_{r-2})\cap \mathcal C_1(d_1',\dots,d_{r-2}')=\varnothing=\mathcal D_1(d_1,\dots,d_{r-2})\cap \mathcal D_1(d_1',\dots,d_{r-2}')$ for all $(d_1,\dots,d_{r-2})\neq (d_1',\dots,d_{r-2}')$. So, it is sufficient to show that $\vert \mathcal D_1(d_1,\dots,d_{r-2})\vert\leq \vert \mathcal C_1(d_1,\dots,d_{r-2})\vert$.

If $m-\sum_{i=1}^{r-2} d_i<0$, then  $\mathcal D_1(d_1,\dots,d_{r-2})=\mathcal C_1(d_1,\dots,d_{r-2})=\varnothing$. Suppose $m-\sum_{i=1}^{r-2} d_i=0$. Then
\begin{equation}
\mathcal C_1(d_1,\dots,d_{r-2})=\left\{ \left(d_1,\dots,d_{r-2},0,0  \right) \right\}.\notag
\end{equation}
If $y_{r-1},y_r>0$, then $\mathcal D_1(d_1,\dots,d_{r-2})=\varnothing$. If $y_{r-1}=0$ and $y_r>0$, or $y_{r-1}>0$ and $y_r=0$, then $\mathcal D_1(d_1,\dots,d_{r-2})=\mathcal C_1(d_1,\dots,d_{r-2})$.

Suppose $m-\sum_{i=1}^{r-2} d_i>0$. Then
\begin{equation}
\mathcal C_1(d_1,\dots,d_{r-2})=\left\{ \left(d_1,\dots,d_{r-2},0,m-\sum_{i=1}^{r-2} d_i  \right), \left(d_1,\dots,d_{r-2},m-\sum_{i=1}^{r-2} d_i ,0 \right)\right\}.\notag
\end{equation}
If $y_{r-1},y_r>m-\sum_{i=1}^{r-2} d_i$, then $\mathcal D_1(d_1,\dots,d_{r-2})=\varnothing$. If $y_{r-1}\leq m-\sum_{i=1}^{r-2} d_i$ and $y_r>m-\sum_{i=1}^{r-2} d_i$, or $y_{r-1}>m-\sum_{i=1}^{r-2} d_i$ and $y_r\leq m-\sum_{i=1}^{r-2} d_i$, then $\vert \mathcal D_1(d_1,\dots,d_{r-2})\vert=1$.  If $y_{r-1}+y_r= m-\sum_{i=1}^{r-2} d_i$, then $\vert \mathcal D_1(d_1,\dots,d_{r-2})\vert=1$.  If $y_{r-1}+y_r<m-\sum_{i=1}^{r-2} d_i$, then $\vert \mathcal D_1(d_1,\dots,d_{r-2})\vert=2$.

In either case, $\vert \mathcal D_1\vert \leq \vert \mathcal C_1\vert$.
\end{proof}

\begin{lm}\label{lm_pre_intersect} Let $m$ be a positive integer, $r\geq 3$, and $k\in [r-2]$. Let $y_1, y_2,\dots, y_k\in \mathbb N_0$ and
\begin{align}
\mathcal C & =\{ \mathbf u\in P(m,r)\ :\ \mathbf u(i)= y_i\ \textnormal{for all $i\in [k-1]$ and}\ \mathbf u(k)=0\},\notag\\
\mathcal D & =\{ \mathbf u\in P(m,r)\ :\ \mathbf u(i)=y_i\ \textnormal{for all $i\in [k-1]$ and}\ \mathbf u(k)=y_k\}.\notag
\end{align}
Then
\begin{itemize}
\item[\textnormal{(a)}] $\mathcal D=\mathcal C$ if ~$m-\sum_{i=1}^{k-1} y_i<0$ or $y_k=0$;
\item[\textnormal{(b)}] $\vert \mathcal D\vert<\vert \mathcal C\vert$ if ~$m-\sum_{i=1}^{k-1} y_i\geq 0$ and $y_k>0$.
\end{itemize}
\end{lm}

\begin{proof} If $y_k=0$, then $\mathcal D=\mathcal C$. If $m-\sum_{i=1}^{k-1} y_i<0$, then $\mathcal D=\mathcal C=\varnothing$. Suppose $m-\sum_{i=1}^{k-1} y_i\geq 0$ and $y_k>0$. Note that
\begin{align}
\vert\mathcal C\vert & =\binom{m-\sum_{i=1}^{k-1} y_i+r-k-1}{r-k-1}.\notag
\end{align}
If $m-\sum_{i=1}^{k} y_i<0$, then $\mathcal D=\varnothing$. So, $\vert \mathcal D\vert<\vert \mathcal C\vert$. If $m-\sum_{i=1}^{k} y_i\geq 0$, then
\begin{align}
\vert\mathcal D\vert & =\binom{m-\sum_{i=1}^{k} y_i+r-k-1}{r-k-1}\notag\\
& <\binom{m-\sum_{i=1}^{k-1} y_i+r-k-1}{r-k-1}=\vert\mathcal C\vert.\notag
\end{align}
\end{proof}

\begin{lm}\label{lm_pre_intersect2} Let $m$ be a positive integer, $r\geq 3$,  $k\in [r-1]$, $k_0\in [k]$, and $m\geq k$.  Let $y_1, y_2,\dots, y_k\in \mathbb N_0$ and
\begin{align}
\mathcal C & =\{ \mathbf u\in P(m,r)\ :\ \mathbf u(i)= y_i\ \textnormal{for some $i\in [k]\setminus \{k_0\}$ or}\ \mathbf u(k_0)=0\},\notag\\
\mathcal D & =\{ \mathbf u\in P(m,r)\ :\ \mathbf u(i)=y_i\ \textnormal{for some $i\in [k]$}\}.\notag
\end{align}
Then
\begin{itemize}
\item[\textnormal{(a)}] $\mathcal D=\mathcal C$ if ~$y_{k_0}=0$;
\item[\textnormal{(b)}] $\vert \mathcal D\vert<\vert \mathcal C\vert$ if ~$y_{k_0}>0$.
\end{itemize}
\end{lm}

\begin{proof} If $k=1$, then $k_0=1$ and
\begin{align}
\mathcal C & =\{ \mathbf u\in P(m,r)\ :\ \mathbf u(1)=0\},\notag\\
\mathcal D & =\{ \mathbf u\in P(m,r)\ :\ \mathbf u(1)=y_1\}.\notag
\end{align}
By Lemma \ref{lm_pre_intersect}, the lemma holds. Suppose $k\geq 2$. By relabelling if necessary, we may assume that $k_0=k$. Let
\begin{align}
\mathcal E_0 & =\{ \mathbf u\in P(m,r)\ :\ \mathbf u(i)=y_i\ \textnormal{for some $i\in [k-1]$}\}\notag\\
\mathcal C_1 & =\{ \mathbf u\in P(m,r)\ :\ \mathbf u(i)\neq y_i\ \textnormal{for all $i\in [k-1]$ and}\ \mathbf u(k)=0\}\notag\\
\mathcal D_1 & =\{ \mathbf u\in P(m,r)\ :\ \mathbf u(i)\neq y_i\ \textnormal{for all $i\in [k-1]$ and}\ \mathbf u(k)=y_k\}.\notag
\end{align}
Then $\mathcal C=\mathcal E_0 \cup \mathcal C_1$ and $\mathcal D=\mathcal E_0 \cup \mathcal D_1$. Note that $\mathcal E_0\cap \mathcal C_1=\varnothing=\mathcal E_0\cap \mathcal D_1$.  If $y_k=0$, then $\mathcal D_1= \mathcal C_1$ and  $\mathcal D=\mathcal C$. Suppose $y_k>0$. It is sufficient to show that $\vert \mathcal D_1\vert <\vert \mathcal C_1\vert$.

 Let $S_i=[m]\setminus \{y_i\}$ and $S=S_1\times S_2\times \cdots S_{k-1}$. For each $(d_1,\dots,d_{k-1})\in  S$, let
\begin{align}
\mathcal C_1(d_1,\dots,d_{k-1}) & =\{ \mathbf u\in P(m,r)\ :\ \mathbf u(i)=d_i\ \textnormal{for all $i\in [k-1]$ and}\ \mathbf u(k)=0\}\notag\\
\mathcal D_1(d_1,\dots,d_{k-1}) & =\{ \mathbf u\in P(m,r)\ :\ \mathbf u(i)=d_i\ \textnormal{for all $i\in [k-1]$ and}\ \mathbf u(k)=y_k\}.\notag
\end{align}
Note that
\begin{align}
\mathcal C_1 & =\bigcup_{(d_1,\dots,d_{k-1})\in  S}\mathcal C_1(d_1,\dots,d_{k-1})\notag\\
\mathcal D_1 & =\bigcup_{(d_1,\dots,d_{k-1})\in  S}\mathcal D_1(d_1,\dots,d_{k-1}),\notag
\end{align}
Furthermore, $\mathcal C_1(d_1,\dots,d_{k-1})\cap \mathcal C_1(d_1',\dots,d_{k-1}')=\varnothing=\mathcal D_1(d_1,\dots,d_{k-1})\cap \mathcal D_1(d_1',\dots,d_{k-1}')$ for all $(d_1,\dots,d_{k-1})\neq (d_1',\dots,d_{k-1}')$.

Suppose $k\in [r-2]$. By  Lemma \ref{lm_pre_intersect}, $\vert \mathcal D_1(d_1,\dots,d_{k-1})\vert\leq \vert \mathcal C_1(d_1,\dots,d_{k-1})\vert$ for all $(d_1,\dots,d_{k-1})\in S$. It remains to show that $\vert \mathcal D_1(d_1,\dots,d_{k-1})\vert< \vert \mathcal C_1(d_1,\dots,d_{k-1})\vert$ for at least one $(d_1,\dots,d_{k-1})\in S$. For each $i\in [k-1]$, set $z_i=0$ if $y_i\neq 0$ and $z_i=1$ if $y_i=0$. Then $(z_1,\dots, z_{k-1})\in S$. Since $m-\sum_{i=1}^{k-1} z_i\geq m-(k-1)> 0$, by Lemma \ref{lm_pre_intersect}, $\vert \mathcal D_1(z_1,\dots,z_{k-1})\vert< \vert \mathcal C_1(z_1,\dots,z_{k-1})\vert$.

Suppose $k=r-1$. If $m-\sum_{i=1}^{r-2} d_i<0$, then  $\mathcal D_1(d_1,\dots,d_{r-2})=\mathcal C_1(d_1,\dots,d_{r-2})=\varnothing$. If $m-\sum_{i=1}^{r-2} d_i=0$, then  $\vert\mathcal C_1(d_1,\dots,d_{r-2})\vert=1$ and $\mathcal D_1(d_1,\dots,d_{r-2})=\varnothing$. If $m-\sum_{i=1}^{r-2} d_i>0$, then  $\vert\mathcal C_1(d_1,\dots,d_{r-2})\vert=1$ and $\vert \mathcal D_1(d_1,\dots,d_{r-2})\vert\leq 1$. So, it remains to show that  $m-\sum_{i=1}^{r-2} d_i=0$ for at least one  $(d_1,\dots,d_{k-1})\in S$.

Note that $m\geq k\geq 2$. Suppose $y_i\neq 0$ for all $i\in [r-2]$. If $y_{i_0}\neq m$ for some $i_0\in [r-2]$, then set $d_i=0$ for all $i\in [r-2]\setminus\{i_0\}$ and $d_{i_0}=m$. If $y_i=m$ for all $i\in [r-2]$, then set $d_1=1$, $d_2=m-1$, and $d_i=0$ for all $i\in [r-2]\setminus \{1,2\}$. Let $U\subseteq [r-2]$ and $U\neq\varnothing$. Suppose that $y_i=0$ for all $i\in U$ and $y_i\neq 0$ for all $i\in [r-2]\setminus U$. Set $d_i=0$ for all $i\in [r-2]\setminus U$. Let $a\in U$. Set $d_i=1$ for all $i\in U\setminus \{a\}$ and $d_a=m-\sum_{i\in U\setminus \{a\}} d_i\geq m-(r-3)>0$.  In either case,  $m-\sum_{i=1}^{r-2} d_i=0$ and  $(d_1,\dots,d_{k-1})\in S$.
\end{proof}

\begin{thm}\label{thm_intersect_l} Let $y_1,\dots, y_r\in\mathbb N_0$ and $r\geq 3$. Let  $S\subseteq [r]$, $S\neq\varnothing$, $m\geq \vert S\vert$,  and
\begin{align}
\mathcal C & =\{ \mathbf u\in P(m,r)\ :\ \mathbf u(i)= 0\ \textnormal{for some $i\in S$}\},\notag\\
\mathcal D & =\{ \mathbf u\in P(m,r)\ :\ \mathbf u(i)= y_i\ \textnormal{for some $i\in S$}\}.\notag
\end{align}
If ~$\sum_{s\in S} y_s>0$, then  $\vert \mathcal D\vert<\vert \mathcal C\vert$.
\end{thm}

\begin{proof} Without loss of generality, we may assume that $S=[k]$ for some $k\in [r]$.
By relabelling if necessary, we may assume that $y_1\geq y_2\geq \cdots\geq y_k$. Since $\sum_{j=1}^k y_j>0$, $y_1>0$. If $k=1$, then by Lemma \ref{lm_pre_intersect}, the theorem holds. Suppose $k\geq 2$.

We shall distinguish two cases.

\vskip 0.5cm
\noindent
{\bf Case 1.} Suppose $k\in [r-1]$. For each $j\in [k-1]$, let
\begin{equation}
\mathcal F_j =\{ \mathbf u\in P(m,r)\ :\ \mathbf u(i)=y_i\ \textnormal{for some $i\in [j]$ or}\ \mathbf u(i)= 0\ \textnormal{for some $j+1\leq i\leq k$}\}.\notag
\end{equation}
By Lemma \ref{lm_pre_intersect2}, $\vert \mathcal F_1\vert<\vert\mathcal C\vert$. Again, by Lemma \ref{lm_pre_intersect2},
\begin{equation}
\vert\mathcal D\vert\leq \vert \mathcal F_{k-1}\vert\leq \cdots \leq \vert \mathcal F_{2}\vert\leq \vert \mathcal F_{1}\vert.\notag
\end{equation}
Hence, $\vert \mathcal D\vert<\vert \mathcal C\vert$.

\vskip 0.5cm
\noindent
{\bf Case 2.} Suppose $k=r$. For each $j\in [r-2]$, let
\begin{equation}
\mathcal F_j =\{ \mathbf u\in P(m,r)\ :\ \mathbf u(i)=y_i\ \textnormal{for some $i\in [j]$ or}\ \mathbf u(i)= 0\ \textnormal{for some $j+1\leq i\leq r$}\}.\notag
\end{equation}
By Lemma \ref{lm_pre_intersect2}, $\vert \mathcal F_1\vert<\vert\mathcal C\vert$. Again, by Lemma \ref{lm_pre_intersect2},
\begin{equation}
 \vert \mathcal F_{r-2}\vert\leq \cdots \leq \vert \mathcal F_{2}\vert\leq \vert \mathcal F_{1}\vert.\notag
\end{equation}
By Lemma \ref{lm_pre_intersect3}, $\vert \mathcal D\vert\leq \vert \mathcal F_{r-2}\vert$. Hence,  $\vert \mathcal D\vert<\vert \mathcal C\vert$.
\end{proof}

Let $\mathbf u=(u_1,u_2,\dots, u_l)\in P(n,l)$.  We define $R(i,\mathbf u)$ to be the element obtained from $\mathbf u$ by removing the $i$-th coordinate, i.e.,
\begin{equation}
R(i;\mathbf u)=(u_1,u_2,\dots, u_{i-1}, u_{i+1},\dots, u_l).\notag
\end{equation}
Inductively, if $x_1,x_2,\dots, x_t$ are distinct elements in $[l]$ with $x_1<x_2<\cdots<x_t$, we define
\begin{equation}
R(x_1,x_2,\dots, x_t;\mathbf u)=R(x_1,x_2,\dots, x_{t-1};R(x_t;\mathbf u)).\notag
\end{equation}
In other words, $R(x_1,x_2,\dots, x_t;\mathbf u)$ is the element obtained from $\mathbf u$ by removing the coordinates $x_{i}$.

\begin{cor}\label{cor_intersect_l} Let $w_1,\dots, w_t,y_{t+1},\dots, y_{r+t}\in\mathbb N_0$, $t\geq 1$ and $r\geq 3$. Let  $S\subseteq [r+t]\setminus [t]$, $S\neq\varnothing$, $m\geq \vert S\vert+\sum_{1\leq i\leq t} w_i$,  and
\begin{align}
\mathcal D & =\{ \mathbf u\in P(m,r+t)\ :\ \mathbf u(i)= w_i\ \textnormal{for all $i\in [t]$ and}\ \mathbf u(i)= y_i\ \textnormal{for some $i\in S$}\}.\notag
\end{align}
If ~$\sum_{s\in S} y_s>0$, then
\begin{equation}
\vert\mathcal D\vert<\sum_{0\leq d\leq \vert S\vert-1}\binom{m-\sum_{1\leq i\leq t} w_i-d+r-2}{r-2}.\notag
\end{equation}
\end{cor}

\begin{proof} Let
\begin{equation}
\mathcal C =\{ \mathbf u\in P(m,r+t)\ :\ \mathbf u(i)= w_i\ \textnormal{for all $i\in [t]$ and}\ \mathbf u(i)= 0\ \textnormal{for some $i\in S$}\}.\notag
\end{equation}
Let $S^*  =\{ a-t\ :\ a\in S\}$, $y_i^* =y_{i+t}$ for all $i\in S^*$,
\begin{align}
\mathcal C^* & =\left\{ \mathbf u\in P\left(m-\sum_{1\leq i\leq t} w_i,r \right)\ :\ \mathbf u(i)= 0\ \textnormal{for some $i\in S^*$}\right\},\notag\\
\mathcal D^* & =\left\{ \mathbf u\in P\left(m-\sum_{1\leq i\leq t} w_i,r \right)\ :\ \mathbf u(i)= y_i^*\ \textnormal{for some $i\in S^*$}\right\}.\notag
\end{align}
By Theorem \ref{thm_intersect_l}, $\vert \mathcal D^*\vert<\vert \mathcal C^*\vert$. For each $\mathbf u\in P(m,r+t)$, the mapping
\begin{equation}
\mathbf u\to R(1,\dots,t;\mathbf u),\notag
\end{equation}
is a bijection from $\mathcal C$ onto $\mathcal C^*$ and from $\mathcal D$ onto $\mathcal D^*$. Thus, $\vert \mathcal D\vert<\vert \mathcal C\vert$.

Let
\begin{equation}
\mathcal F=\{ \mathbf u\in P(m,r+t)\ :\ \mathbf u(i)= w_i\ \textnormal{for all $i\in [t]$}\}.\notag
\end{equation}
Then
\begin{equation}
\mathcal C=\mathcal F\setminus \{\mathbf u\in P(m,r+t)\ :\ \mathbf u(i)\neq 0\ \textnormal{for all $i\in S$}\},\notag
\end{equation}
and
\begin{align}
\vert\mathcal C\vert &=\binom{m-\sum_{1\leq i\leq t} w_i+r-1}{r-1}-\binom{m-\sum_{1\leq i\leq t} w_i-\vert S\vert+r-1}{r-1}\notag\\
&=\sum_{0\leq d\leq \vert S\vert-1}\binom{m-\sum_{1\leq i\leq t} w_i-d+r-2}{r-2}.\notag
\end{align}
Hence, the corollary holds.
\end{proof}

\section{Main result}

A family $\mathcal B\subseteq P(n,l)$ is said to be \emph{independent} if $I(\mathbf u,\mathbf v)=\varnothing$, i.e., $\vert I(\mathbf u,\mathbf v)\vert=0$, for all $\mathbf u,\mathbf v\in \mathcal B$ with $\mathbf u\neq \mathbf v$.

We shall need the following theorem  \cite[Theorem 2.3]{Ku_Wong4}.

\begin{thm}\label{thm_independent} Let $m,n$ be positive integers satisfying $m\leq n$, and let $q,r,s$ be positive integers with $r,s\geq 2$ and $n\geq (2s)^{2^{r-2}q}+1$. If $\mathcal A\subseteq P(m,r)$ such that $\vert\mathcal A\vert\geq n^{\frac{1}{q}}\binom{n+r-2}{r-2}$,
then there is an independent set $\mathcal B\subseteq \mathcal A$ with $\vert \mathcal B\vert\geq s+1$.
\end{thm}

Let $\mathcal A\subseteq P(n,l)$. Let $x_1,x_2,\dots, x_t$ be distinct elements in $[l]$ with $x_1<x_2<\cdots<x_t$, and $y_1,y_2,\dots, y_t\in \mathbb N_0$ with $\sum_{j=1}^t y_j\leq n$. We set
\begin{align}
\mathcal A(x_1,x_2,\dots, x_t;y_1,y_2,\dots, y_t) & =\{ \mathbf u\in \mathcal A\ :\ \mathbf u(x_i)=y_i\ \textnormal{for all $i$}\},\notag\\
\mathcal A^*(x_1,x_2,\dots, x_t;y_1,y_2,\dots, y_t) & =\{ R(x_1,x_2,\dots, x_t;\mathbf u)\ :\ \mathbf u\in \mathcal A(x_1,x_2,\dots, x_t;y_1,y_2,\dots, y_t)\}.\notag
\end{align}
Note that
\begin{equation}
\mathcal A^*(x_1,x_2,\dots, x_t;y_1,y_2,\dots, y_t)\subseteq P\left (n-\sum_{j=1}^t y_j,l-t\right ),\notag
\end{equation}
and
\begin{equation}
\vert \mathcal A^*(x_1,x_2,\dots, x_t;y_1,y_2,\dots, y_t)\vert=\vert \mathcal A(x_1,x_2,\dots, x_t;y_1,y_2,\dots, y_t)\vert.\notag
\end{equation}

\begin{lm}\label{lm_main_independent} Let $\mathcal A\subseteq P(n,l)$ be $t$-intersecting and $l\geq t+3$. Let $x_1,x_2,\dots, x_{t+1}$ be distinct elements in $[l]$ with $x_1<x_2<\cdots<x_{t+1}$, and $y_1,y_2,\dots, y_{t+1}\in \mathbb N_0$ with $\sum_{j=1}^{t+1} y_j\leq n$. If $\mathcal A^*(x_1,\dots, x_{t+1};y_1,\dots, y_{t+1})$ has an independent set of size at least $l-t$, then
\begin{equation}
\mathcal A\subseteq \bigcup_{\substack{T\subseteq [t+1],\\ \vert T\vert=t}}\left \{ \mathbf u\in P(n,l)\ :\  \mathbf u(x_s)=y_s\ \textnormal{for all $s\in T$}\right\}.\notag
\end{equation}
\end{lm}

\begin{proof} Let
\begin{equation}
\mathcal B=\{ R(x_1,x_2,\dots, x_{t+1};\mathbf u_i)\ :\ i=1,2,\dots, l-t\},\notag
\end{equation}
be an independent set of size $l-t$ in $\mathcal A^*(x_1,\dots, x_{t+1};y_1,\dots, y_{t+1})$. Here each $\mathbf u_i$ is an element of $\mathcal A$ such that ${\bf u}_{i}(x_{j}) = y_{j}$ for all $1 \le j \le t$.

Let $\mathbf v\in\mathcal A$. Suppose
\begin{equation}
\mathbf v\notin \bigcup_{\substack{T\subseteq [t+1],\\ \vert T\vert=t}}\left \{ \mathbf u\in P(n,l)\ :\  \mathbf u(x_s)=y_s\ \textnormal{for all $s\in T$}\right\}.\notag
\end{equation}
Then there exist $j_1$ and $j_2$ with $1\leq j_1<j_2\leq t+1$ such that $\mathbf v(x_{j_1})\neq y_{j_1}$  and $\mathbf v(x_{j_2})\neq y_{j_2}$. Since $\mathcal A$ is $t$-intersecting,
\begin{equation}
\vert I(R(x_1,x_2,\dots, x_{t+1};\mathbf v), R(x_1,x_2,\dots, x_{t+1};\mathbf u_i))\vert\geq 1,\notag
\end{equation}
for $i=1,2,\dots, l-t$.

Set $\mathbf z=R(x_1,x_2,\dots, x_{t+1};\mathbf v)$ and $\mathbf w_i=R(x_1,x_2,\dots, x_{t+1};\mathbf u_i)$. Since $\mathcal B$ is independent,
\begin{equation}
I(\mathbf z,\mathbf w_i)\cap I(\mathbf z,\mathbf w_{i'})=\varnothing,\notag
\end{equation}
for $i\neq i'$. Therefore $\left\vert \bigcup_{i=1}^{l-t} I(\mathbf z,\mathbf w_i) \right\vert  =\sum_{i=1}^{l-t} \vert I(\mathbf z,\mathbf w_i)\vert\geq \sum_{i=1}^{l-t} 1=l-t$,
but on the other hand, $\bigcup_{i=1}^{l-t} I(\mathbf z,\mathbf w_i)\subseteq [l] \setminus \{x_{j}: j \in [t+1]\}$
which is of size at most $l-t-1$, a contradiction. Hence, \begin{equation}
\mathbf v\in \bigcup_{\substack{T\subseteq [t+1],\\ \vert T\vert=t}}\left \{ \mathbf u\in P(n,l)\ :\  \mathbf u(x_s)=y_s\ \textnormal{for all $s\in T$}\right\},\notag
\end{equation}
and the lemma follows.
\end{proof}

\begin{lm}\label{lm_pre_inequality} If ~$x_1,x_2,\dots, x_r$ are positive real numbers, then
\begin{equation}
\prod_{i=1}^r (1+x_i)-\prod_{i=1}^r(1-x_r)\geq 2\sum_{i=1}^r x_i.\notag
\end{equation}
\end{lm}

\begin{proof} Let $S\subseteq [r]$. If $\vert S\vert$ is even, then $\prod_{s\in S} x_s$ appears in the expansions of $\prod_{i=1}^r (1+x_i)$ and $\prod_{i=1}^r(1-x_r)$. So, $\prod_{s\in S} x_s$ does not appear in $\prod_{i=1}^r (1+x_i)-\prod_{i=1}^r(1-x_r)$. If $\vert S\vert$ is odd, then $\prod_{s\in S} x_s$ and $-\prod_{s\in S} x_s$ appear in the expansions of $\prod_{i=1}^r (1+x_i)$ and $\prod_{i=1}^r(1-x_r)$, respectively. So, $2\prod_{s\in S} x_s$ appears in $\prod_{i=1}^r (1+x_i)-\prod_{i=1}^r(1-x_r)$. The lemma follows by just considering all subsets  $S$ of $[r]$ with $\vert S\vert=1$.
\end{proof}

\begin{lm}\label{lm_inequality} Let $m,n$ be positive integers with $n\geq m+1$. Then
\begin{equation}
{n+m \choose m}-{n-1 \choose m}\geq \frac{(m+1)n^{m-1}}{(m-1)!}.\notag
\end{equation}
\end{lm}

\begin{proof} Note that
\begin{equation}
{n+m \choose m}=\frac{n^{m}}{m!}\prod_{i=1}^{m}\left(1+\frac{i}{n} \right),\notag
\end{equation}
and
\begin{equation}
{n-1 \choose m}=\frac{n^{m}}{m!}\prod_{i=1}^{m}\left(1-\frac{i}{n} \right),\notag
\end{equation}
It then follows from Lemma \ref{lm_pre_inequality} that
\begin{align}
{n+m\choose m}-{n-1 \choose m} & \geq \frac{n^{m}}{m!}\left(2\sum_{i=1}^{m} \frac{i}{n}\right)\notag\\
& = \frac{(m+1)n^{m-1}}{(m-1)!}.\notag
\end{align}
\end{proof}

\begin{lm}\label{lm_upper_bound} Let $m,n$ be positive integers with $n\geq m$. Then
\begin{equation}
\frac{n^{m}}{m!}<{n+m \choose m}<\frac{n^{m}}{m!}\left(1+\frac{2^mm}{n} \right).\notag
\end{equation}
\end{lm}

\begin{proof} Note that
\begin{equation}
{n+m\choose m} =\frac{n^{m}}{m!}\prod_{i=1}^{m}\left(1+\frac{i}{n} \right)>\frac{n^{m}}{m!}.\notag
\end{equation}
For the second inequality, it is sufficient to show that
\begin{equation}
\prod_{i=1}^{m}\left(1+\frac{i}{n} \right)<\left(1+\frac{2^mm}{n} \right).\notag
\end{equation}
Now,
\begin{align}
\prod_{i=1}^{m}\left(1+\frac{i}{n} \right)& \leq  \left(1+\frac{m}{n} \right)^m\notag\\
& = \left(1+\sum_{i=1}^{m}  {m\choose i}    \left(\frac{m}{n}\right)^i \right)\notag\\
& \leq\left(1+\frac{m}{n}\left(\sum_{i=1}^{m}  {m\choose i} \right) \right)\notag\\
& < \left(1+\frac{2^mm}{n} \right).\notag
\end{align}
\end{proof}

\begin{lm}\label{lm_upper_bound2} Let $m$ be a positive integer and $f,g$ be positive real numbers. There exists a constant $n_0=n_0(f,g,m)$ depending on $f,g$ and $m$ such that if $n\geq n_0$, then
\begin{equation}
f{n+m \choose m}+g\sqrt n{n+m-1 \choose m-1}<\frac{n^{m}(f+1)}{m!}.\notag
\end{equation}
\end{lm}

\begin{proof} Suppose $n\geq \max\left( m, 2^{m-1}(m-1), \left(\frac{2^m}{g}\right)^2, (3fgm)^2  \right)$. By Lemma \ref{lm_upper_bound},
\begin{align}
&f{n+m \choose m}+g\sqrt n{n+m-1 \choose m-1}\notag\\
&<f\frac{n^{m}}{m!}\left(1+\frac{2^mm}{n} \right)+g\sqrt n\frac{n^{m-1}}{(m-1)!}\left(1+\frac{2^{m-1}(m-1)}{n} \right)\notag\\
&\leq f\frac{n^{m}}{m!}\left(1+\frac{2^mm}{n} \right)+\frac{2gn^{m-\frac{1}{2}}}{(m-1)!}\notag\\
&=f\frac{n^{m}}{m!}\left(1+\frac{2^mm}{n} +\frac{2gm}{\sqrt n}\right)\notag\\
&\leq f\frac{n^{m}}{m!}\left(1+\frac{3gm}{\sqrt n}\right)\notag\\
&=\frac{n^{m}}{m!}\left(f+\frac{3fgm}{\sqrt n}\right)\notag\\
&<\frac{n^{m}(f+1)}{m!}.\notag
\end{align}
\end{proof}

\begin{proof}[Proof of Theorem \ref{thm_main}] We may assume that $\mathcal A$ is maximal $t$-intersecting in the sense that
$\mathcal A\cup \{\mathbf u\}$ is not $t$-intersecting for any $\mathbf u\in P(n,l)\setminus \mathcal A$.

Suppose
\begin{align}
\vert \mathcal A\vert< \frac{n^{l-t-2}(t+3)}{(l-t-2)!}.\notag
\end{align}
By Lemma \ref{lm_inequality},
\begin{equation}
{n+l-t-1 \choose l-t-1}-{n-1 \choose l-t-1}\geq \frac{(l-t)n^{l-t-2}}{(l-t-2)!}.\notag
\end{equation}
Since $l\geq 2t+3$, $\vert \mathcal A\vert< {n+l-t-1 \choose l-t-1}-{n-1 \choose l-t-1}$. So, we may assume that
\begin{align}\label{eq_1}
\vert \mathcal A\vert\geq  \frac{n^{l-t-2}(t+3)}{(l-t-2)!}.\tag{1}
\end{align}

Let $\mathbf w=(w_1,w_2,\dots, w_l)\in\mathcal A$ be fixed. Then
\begin{equation}
\mathcal A=\bigcup_{\substack{\{x_1,x_2,\dots,x_t\}\subseteq [l],\\ x_1<x_2<\cdots<x_t}} \mathcal A(x_1,x_2,\dots, x_t;w_{x_1},w_{x_2},\dots, w_{x_t}).\notag
\end{equation}
Let $\{x_1',x_2',\dots,x_t'\}\subseteq [l]$ with $x_1'<x_2'<\cdots<x_t'$ be fixed and
\begin{equation}
\mathcal C=\mathcal A(x_1',x_2',\dots, x_t';w_{x_1'},w_{x_2'},\dots, w_{x_t'}).\notag
\end{equation}
We may assume that $\vert \mathcal C\vert$ is maximum in the sense that
\begin{equation}
\vert \mathcal A(x_1,x_2,\dots, x_t;w_{x_1},w_{x_2},\dots, w_{x_t})\vert\leq \vert \mathcal C\vert,\notag
\end{equation}
for all $\{x_1,x_2,\dots,x_t\}\subseteq [l]$ with $x_1<x_2<\cdots<x_t$.

By relabelling if necessary, we may assume that $x_i'=i$ for all $i$, i.e.,
\begin{equation}
\mathcal C=\mathcal A(1,2,\dots, t;w_{1},w_{2},\dots, w_{t}).\notag
\end{equation}
Let $P=P(n,l)$. Since $\mathcal A$ is non-trivially $t$-intersecting,
\begin{equation}
\mathcal A\nsubseteq P(1,2,\dots, t;w_{1},w_{2},\dots, w_{t}).\notag
\end{equation}
Therefore there is a $\mathbf y=(y_1,y_2,\dots, y_{l})\in \mathcal A$ with $y_i\neq w_i$ for some $1\leq i\leq t$. Since $\mathcal C\cup \{\mathbf y\}$ is $t$-intersecting,
\begin{equation}
\mathcal C=\bigcup_{t+1\leq s\leq l} \mathcal A(1,2,\dots, t,s;w_{1},w_{2},\dots, w_{t},y_s).\notag
\end{equation}

We shall distinguish 5 cases.

\vskip 0.5cm
\noindent
{\bf Case 1}. Suppose that
\begin{equation}
 \vert \mathcal A(1,2,\dots, t,s;w_{1},w_{2},\dots, w_{t},y_s)\vert\leq n^{\frac{1}{2}}\binom{n+l-t-3}{l-t-3},\notag
\end{equation}
for all $t+1\leq s\leq l$. Then
\begin{align}
\vert \mathcal C\vert& \leq (l-t) n^{\frac{1}{2}}\binom{n+l-t-3}{l-t-3}.\notag
\end{align}
By the  maximality of $\vert \mathcal C\vert$,
\begin{align}
\vert \mathcal A\vert& \leq \binom{l}{t}(l-t) n^{\frac{1}{2}}\binom{n+l-t-3}{l-t-3}.\notag
\end{align}
By Lemma \ref{lm_upper_bound2},
\begin{align}
\vert \mathcal A\vert< \frac{n^{l-t-2}}{(l-t-2)!},\notag
\end{align}
contradicting equation (\ref{eq_1}).  Hence,  Case 1 cannot happen.

Now, by relabelling if necessary we may assume that there is a $k\in \{t+1,t+2,\dots, l\}$ such that
\begin{equation}
\vert \mathcal A(1,2,\dots, t,s;w_{1},w_{2},\dots, w_{t},y_s)\vert\geq n^{\frac{1}{2}}\binom{n+l-t-3}{l-t-3},\notag
\end{equation}
for all $t+1\leq s\leq k$ and
\begin{equation}
\vert \mathcal A(1,2,\dots, t,s;w_{1},w_{2},\dots, w_{t},y_s)\vert\leq n^{\frac{1}{2}}\binom{n+l-t-3}{l-t-3},\notag
\end{equation}
for all $k+1\leq  s\leq l$.

\vskip 0.5cm
\noindent
{\bf Case 2}. Suppose $k= t+1$. Since
\begin{equation}
\vert \mathcal A^*(1,2,\dots, t,t+1;w_{1},w_{2},\dots, w_{t},y_{t+1})\vert=\vert \mathcal A(1,2,\dots, t,t+1;w_{1},w_{2},\dots, w_{t},y_{t+1})\vert,\notag
\end{equation}
and
\begin{equation}
\mathcal A^*(1,2,\dots, t,t+1;w_{1},w_{2},\dots, w_{t},y_{t+1})\subseteq P\left (n-y_{t+1}-\sum_{j=1}^t w_{j},l-t-1\right ),\notag
\end{equation}
by Theorem \ref{thm_independent}, $\mathcal A^*(1,2,\dots, t,t+1;w_{1},w_{2},\dots, w_{t},y_{t+1})$ has an independent set of size at least $l-t$, if
$n\geq (2(l-t-1))^{2^{l-t-2}}+1$. Then it follows from Lemma \ref{lm_main_independent} that
\begin{equation}
\mathcal A=\mathcal C\cup \bigcup_{1\leq j\leq t}\mathcal A(1,\dots,j-1,j+1,\dots ,t,t+1;w_{1},\dots,w_{j-1},w_{j+1},\dots, w_{t},y_{t+1}).\notag
\end{equation}
By the choice of $\vert \mathcal C\vert$, $\vert\mathcal A\vert\leq (t+1)\vert \mathcal C\vert$.

Now,
\begin{align}
\vert \mathcal A(1,2,\dots, t,t+1;w_{1},w_{2},\dots, w_{t},y_{t+1})\vert &\leq \binom{n-y_{t+1}-(\sum_{j=1}^t w_j)+l-t-2}{l-t-2}\notag\\
&\leq \binom{n+l-t-2}{l-t-2}.\notag
\end{align}
Therefore
\begin{align}
\vert \mathcal C\vert &\leq \binom{n+l-t-2}{l-t-2}+(l-t-1)n^{\frac{1}{2}}\binom{n+l-t-3}{l-t-3},\notag
\end{align}
and
\begin{align}
\vert \mathcal A\vert &\leq (t+1)\binom{n+l-t-2}{l-t-2}+(t+1)(l-t-1)n^{\frac{1}{2}}\binom{n+l-t-3}{l-t-3}.\notag
\end{align}
By Lemma \ref{lm_upper_bound2},
\begin{align}
\vert \mathcal A\vert &<\frac{n^{l-t-2}(t+2)}{(l-t-2)!},\notag
\end{align}
contradicting equation (\ref{eq_1}).  Hence,  Case 2 cannot happen.

\vskip 0.5cm
\noindent
{\bf Case 3}. Suppose $k= t+2$. Again,
by Theorem \ref{thm_independent},  $\mathcal A^*(1,2,\dots, t,s;w_{1},w_{2},\dots, w_{t},y_{s})$ has an independent set of size at least $l-t$ for $s\in\{t+1,t+2\}$, if
$n\geq (2(l-t-1))^{2^{l-t-2}}+1$. For each $j\in [t]$, let
\begin{align}
\mathcal Q_j=\{ \mathbf u\in \mathcal A\ & :\ \mathbf u(i)=w_i\ \textnormal{for all}\ i\in [t]\setminus \{j\}, \mathbf u(j)\neq w_j,\notag\\
& \hskip 1cm \mathbf u(t+1)=y_{t+1}\ \textnormal{and}\ \mathbf u(t+2)=y_{t+2}\}.\notag
\end{align}
It follows from Lemma \ref{lm_main_independent} that
\begin{equation}
\mathcal A=\mathcal C\cup \bigcup_{1\leq j\leq t}\mathcal Q_j.\notag
\end{equation}
Let $n'_j=n-y_{t+1}-y_{t+2}-\left(\sum_{1\leq i\leq t,i\neq j}w_i\right)$. If $n'_j<0$, then $\vert \mathcal Q_j\vert=0$. If $n'_j\geq 0$, then
\begin{align}
\vert \mathcal Q_j\vert &\leq  \sum_{\substack{0\leq d\leq n_j',\\ d\neq w_j}} \binom{n'_j-d+l-t-3}{l-t-3}\notag\\
&\leq \sum_{0\leq d\leq n_j'} \binom{n'_j-d+l-t-3}{l-t-3}\notag\\
& =\binom{n'_j+l-t-2}{l-t-2}\notag\\
&\leq \binom{n+l-t-2}{l-t-2}.\notag
\end{align}
Therefore $\vert \mathcal A\vert\leq \vert \mathcal C\vert+t\binom{n+l-t-2}{l-t-2}$.

Now, for each $s\in\{t+1,t+2\}$,
\begin{align}
\vert \mathcal A(1,2,\dots, t,s;w_{1},w_{2},\dots, w_{t},y_{s})\vert &\leq \binom{n-y_{s}-(\sum_{j=1}^t w_j)+l-t-2}{l-t-2}\notag\\
&\leq \binom{n+l-t-2}{l-t-2}.\notag
\end{align}
Therefore
\begin{align}
\vert \mathcal C\vert &\leq 2\binom{n+l-t-2}{l-t-2}+(l-t-2)n^{\frac{1}{2}}\binom{n+l-t-3}{l-t-3},\notag
\end{align}
and
\begin{align}
\vert \mathcal A\vert &\leq (t+2)\binom{n+l-t-2}{l-t-2}+(l-t-2)n^{\frac{1}{2}}\binom{n+l-t-3}{l-t-3}.\notag
\end{align}
By Lemma \ref{lm_upper_bound2},
\begin{align}
\vert \mathcal A\vert &<\frac{n^{l-t-2}(t+3)}{(l-t-2)!},\notag
\end{align}
contradicting equation (\ref{eq_1}).  Hence,  Case 3 cannot happen.

We may assume that $t+3\leq k\leq  l$.
In particular,
\begin{equation}
\vert \mathcal A(1,2,\dots, t,t+1;w_{1},w_{2},\dots, w_{t},y_{t+1})\vert\geq  n^{\frac{1}{2}}\binom{n+l-t-3}{l-t-3}.\notag
\end{equation}
 On the other hand,
\begin{align}
\vert \mathcal A(1,2,\dots, t,t+1;w_{1},w_{2},\dots, w_{t},y_{t+1})\vert &\leq \binom{n-y_{t+1}-(\sum_{j=1}^t w_j)+l-t-2}{l-t-2}\notag\\
&\leq \binom{n-(\sum_{j=1}^t w_j)+l-t-2}{l-t-2}\notag.\notag
\end{align}

If $n-\sum_{j=1}^t w_j<l-t$, then $n^{\frac{1}{2}}\binom{n+l-t-3}{l-t-3}\leq \binom{n-(\sum_{j=1}^t w_j)+l-t-2}{l-t-2}<\binom{2l-2t-2}{l-t-2}$, which is impossible for large $n$. So, we may assume that
\begin{equation}\label{eq2}
n-\sum_{j=1}^t w_j\geq l-t.\tag{2}
\end{equation}

\vskip 0.5cm
\noindent
{\bf Case 4}. Suppose $t+3\leq k\leq  l-1$. Again,
by Theorem \ref{thm_independent},  $\mathcal A^*(1,2,\dots, t,s;w_{1},w_{2},\dots, w_{t},y_{s})$ has an independent set of size at least $l-t$ for $s\in\{t+1,t+2,t+3\}$, if
$n\geq (2(l-t-1))^{2^{l-t-2}}+1$. For each $j\in [t]$, let
\begin{align}
\mathcal Q_j=\{ \mathbf u\in \mathcal A\ & :\ \mathbf u(i)=w_i\ \textnormal{for all}\ i\in [t]\setminus \{j\}, \mathbf u(j)\neq w_j,\notag\\
& \hskip 1cm \mathbf u(t+1)=y_{t+1}, \mathbf u(t+2)=y_{t+2},\ \textnormal{and}\ \mathbf u(t+3)=y_{t+3}\}.\notag
\end{align}
It follows from Lemma \ref{lm_main_independent} that
\begin{equation}
\mathcal A=\mathcal C\cup \bigcup_{1\leq j\leq t}\mathcal Q_j.\notag
\end{equation}
Let $n'_j=n-y_{t+1}-y_{t+2}-y_{t+3}-\left(\sum_{1\leq i\leq t,i\neq j}w_i\right)$. If $n'_j<0$, then $\vert \mathcal Q_j\vert=0$. If $n'_j\geq 0$, then
\begin{align}
\vert \mathcal Q_j\vert &\leq  \sum_{\substack{0\leq d\leq n_j',\\ d\neq w_j}} \binom{n'_j-d+l-t-4}{l-t-4}\notag\\
&\leq \sum_{0\leq d\leq n_j'} \binom{n'_j-d+l-t-4}{l-t-4}\notag\\
& =\binom{n'_j+l-t-3}{l-t-3}\notag\\
&\leq \binom{n+l-t-3}{l-t-3}.\notag
\end{align}
Therefore $\vert \mathcal A\vert\leq \vert \mathcal C\vert+t\binom{n+l-t-3}{l-t-3}$.

Let
\begin{align}
\mathcal D  =\{ \mathbf u\in P(n,l)\ :\ \mathbf u(i)= w_i\ \textnormal{for all $i\in [t]$ and}\ \mathbf u(i)= y_i\ \textnormal{for some $t+1\leq i\leq l-1$}\}.\notag
\end{align}
By equation (\ref{eq2}) and Corollary \ref{cor_intersect_l},
\begin{align}
\vert \mathcal D\vert &\leq \sum_{0\leq d\leq l-t-2}\binom{n-\sum_{1\leq i\leq t} w_i-d+l-t-2}{l-t-2}\notag\\
&\leq \sum_{0\leq d\leq l-t-2}\binom{n-d+l-t-2}{l-t-2}.\notag
\end{align}
Since
\begin{equation}
\bigcup_{t+1\leq s\leq l-1} \mathcal A(1,2,\dots, t,s;w_{1},w_{2},\dots, w_{t},y_s)\subseteq \mathcal D,\notag
\end{equation}

\begin{equation}
\vert \mathcal C\vert\leq \sum_{0\leq d\leq l-t-2}\binom{n-d+l-t-2}{l-t-2}+n^{\frac{1}{2}}\binom{n+l-t-3}{l-t-3},\notag
\end{equation}
and
\begin{align}
\vert \mathcal A\vert &\leq \sum_{0\leq d\leq l-t-2}\binom{n-d+l-t-2}{l-t-2}+(n^{\frac{1}{2}}+t)\binom{n+l-t-3}{l-t-3}\notag\\
 &\leq \sum_{0\leq d\leq l-t-2}\binom{n-d+l-t-2}{l-t-2}+2n^{\frac{1}{2}}\binom{n+l-t-3}{l-t-3}.\notag
\end{align}
 Now,
\begin{equation}
{n+l-t-1 \choose l-t-1}-{n-1 \choose l-t-1}=\sum_{0\leq d\leq l-t-1}\binom{n-d+l-t-2}{l-t-2}.\notag
\end{equation}
Therefore $\vert \mathcal A\vert<{n+l-t-1 \choose l-t-1}-{n-1 \choose l-t-1}$ if and only if
\begin{equation}
2n^{\frac{1}{2}}\binom{n+l-t-3}{l-t-3}<\binom{n-1}{l-t-2}.\notag
\end{equation}
By Lemma \ref{lm_upper_bound},
\begin{align}
2n^{\frac{1}{2}}\binom{n+l-t-3}{l-t-3}&<\frac{4n^{l-t-\frac{5}{2}}}{(l-t-3)!}\notag\\
&<\frac{n^{l-t-2}}{(l-t-2)!}\left( 1-\frac{l-t-1}{n} \right)^{l-t-2}\notag\\
&<\binom{n-1}{l-t-2}.\notag
\end{align}
Hence, $\vert \mathcal A\vert<{n+l-t-1 \choose l-t-1}-{n-1 \choose l-t-1}$ and Case 4 is done.

\vskip 0.5cm
\noindent
{\bf Case 5}. Suppose $k=l$. Again,
by Theorem \ref{thm_independent},  $\mathcal A^*(1,2,\dots, t,s;w_{1},w_{2},\dots, w_{t},y_{s})$ has an independent set of size at least $l-t$ for $s\in [l]\setminus [t]$, if
$n\geq (2(l-t-1))^{2^{l-t-2}}+1$. For each $j\in [t]$, let
\begin{align}
\mathcal Q_j=\{ \mathbf u\in \mathcal A\ & :\ \mathbf u(i)=w_i\ \textnormal{for all}\ i\in [t]\setminus \{j\}, \mathbf u(j)\neq w_j,\notag\\
& \hskip 1cm \textnormal{and}\ \mathbf u(i)=y_{i}\ \textnormal{for all}\ i\in [l]\setminus [t]\}.\notag
\end{align}
It follows from Lemma \ref{lm_main_independent} that
\begin{equation}
\mathcal A=\mathcal C\cup \bigcup_{1\leq j\leq t}\mathcal Q_j.\notag
\end{equation}
Let $n'_j=n-\sum_{t+1\leq i\leq l} y_{i}-\left(\sum_{1\leq i\leq t,i\neq j}w_i\right)$. If $n'_j<0$, then $\vert \mathcal Q_j\vert=0$. If $n'_j\geq 0$, then
$\vert \mathcal Q_j\vert=1$. Therefore $\vert \mathcal A\vert\leq \vert \mathcal C\vert+t$.

Let
\begin{align}
\mathcal D  =\{ \mathbf u\in P(n,l)\ :\ \mathbf u(i)= w_i\ \textnormal{for all $i\in [t]$ and}\ \mathbf u(i)= y_i\ \textnormal{for some $t+1\leq i\leq l$}\}.\notag
\end{align}
By equation (\ref{eq2}) and Corollary \ref{cor_intersect_l},
\begin{align}
\vert \mathcal D\vert &\leq \sum_{0\leq d\leq l-t-1}\binom{n-\sum_{1\leq i\leq t} w_i-d+l-t-2}{l-t-2}\notag\\
&\leq \sum_{0\leq d\leq l-t-1}\binom{n-d+l-t-2}{l-t-2}.\notag
\end{align}
Since
\begin{equation}
\mathcal C=\bigcup_{t+1\leq s\leq l} \mathcal A(1,2,\dots, t,s;w_{1},w_{2},\dots, w_{t},y_s)\subseteq \mathcal D,\notag
\end{equation}

\begin{equation}
\vert \mathcal C\vert\leq  \sum_{0\leq d\leq l-t-1}\binom{n-d+l-t-2}{l-t-2},\notag
\end{equation}
and
\begin{align}
\vert \mathcal A\vert &\leq \sum_{0\leq d\leq l-t-1}\binom{n-d+l-t-2}{l-t-2}+t.\notag
\end{align}
Now,
\begin{equation}
{n+l-t-1 \choose l-t-1}-{n-1 \choose l-t-1}=\sum_{0\leq d\leq l-t-1}\binom{n-d+l-t-2}{l-t-2}.\notag
\end{equation}
Therefore $\vert \mathcal A\vert\leq {n+l-t-1 \choose l-t-1}-{n-1 \choose l-t-1}+t$. Furthermore, equality holds if and only if $w_i=0$ for all $i\in [t]$ and $y_i=0$ for all $i\in [l]\setminus [t]$. By the maximality of $\mathcal A$,
\begin{equation}
\mathcal C=\bigcup_{s\in [l]\setminus [t]} \mathcal A_s,\notag
\end{equation}
and $\mathcal Q_j=\{ \mathbf  q_j\}$ for all $j\in [t]$. This completes the proof of the theorem.
\end{proof}

\end{document}